\documentclass[12pt]{article}

\usepackage[T1]{fontenc}
\usepackage[utf8]{inputenc}
\usepackage{lmodern}
\usepackage{microtype}
\usepackage{amsmath,amssymb,amsthm,mathtools}
\usepackage{enumitem}
\usepackage{xcolor}
\usepackage{hyperref}
\usepackage[nameinlink,noabbrev]{cleveref}

\hypersetup{
  colorlinks=true,
  linkcolor=blue,
  citecolor=blue,
  urlcolor=blue,
  pdftitle={The L log L-Luxemburg Discrepancy for Nonnegative Weights Suffers from the Curse of Dimensionality},
  pdfauthor={Josef Dick and Friedrich Pillichshammer},
  pdfsubject={Luxemburg discrepancy and the curse of dimensionality},
  pdfkeywords={discrepancy, Orlicz space, Luxemburg norm, L log L, curse of dimensionality, tractability}
}

\newtheorem{theorem}{Theorem}[section]
\newtheorem{corollary}[theorem]{Corollary}
\newtheorem{lemma}[theorem]{Lemma}
\newtheorem{proposition}[theorem]{Proposition}
\newtheorem{remark}[theorem]{Remark}

\newcommand{\E}{\mathbb{E}}
\newcommand{\N}{\mathbb{N}}

\newcommand{\ind}{\mathbf{1}}
\newcommand{\dd}{\,\mathrm{d}}
\newcommand{\disc}{\operatorname{disc}}
\newcommand{\R}{\mathbb{R}}
\newcommand{\norm}[1]{\left\lVert #1\right\rVert}
\newcommand{\bst}{\boldsymbol{t}}
\newcommand{\bss}{\boldsymbol{s}}
\newcommand{\bsx}{\boldsymbol{x}}

\newcommand{\bszero}{\boldsymbol{0}}
\newcommand{\cF}{\mathcal{F}}
\newcommand{\cP}{\mathcal{P}}
\newcommand{\cA}{\mathcal{A}}
\newcommand{\qmc}{{\rm qmc}}

\title{The \texorpdfstring{$L\log L$}{L log L}-Luxemburg Discrepancy for Nonnegative Weights Suffers from the Curse of Dimensionality}
\author{Josef Dick\thanks{J. D. was supported in part by a grant of access to OpenAI models through the ChatGPT for Academic Researchers program.}\ \ and Friedrich Pillichshammer\thanks{F.P. was funded in part by the Austrian Science Fund (FWF) 10.55776/PAT1173425.}}
\date{}

\begin{document}

\maketitle

\begin{abstract}
We study the normalized discrepancy in the Luxemburg norm generated by \(\Phi(s)=s\log(1+s)\). For arbitrary nonnegative quadrature weights, we prove that the associated information complexity suffers from the curse of dimensionality: for every \(\varepsilon\in(0,1)\), there exist constants \(c_\varepsilon>0\) and \(d_0=d_0(\varepsilon)\in\mathbb N\) such that
\[
N_{\Phi,+}(\varepsilon,d)\ge c_{\varepsilon} d^{-1}(1.07735\ldots)^d \qquad \mbox{for all } d\ge d_0.
\]
The result therefore also applies to equal-weight quasi-Monte Carlo rules.
\end{abstract}



\section{Introduction}

Let $\cP=\{\bsx_1,\ldots,\bsx_N\}\subseteq [0,1]^d$ be an $N$-element point set, and let \(\cA=\{a_1,\ldots,a_N\}\) be corresponding real weights. For \(\bst=(t_1,\ldots,t_d)\in[0,1]^d\), define the generalized discrepancy function by 
\[
\Delta_{\cP,\cA}(\bst) =\sum_{k=1}^N a_k \mathbf{1}_{\{\bsx_k<\bst\}}-\prod_{j=1}^d t_j,
\]
where \(\mathbf{1}_{\{\bsx_k<\boldsymbol t\}}\) equals 1 if, for the $j$th component $x_{k,j}$ of $\bsx_k$, we have \(x_{k,j}<t_j\) for every \(j\in\{1,\ldots,d\}\), and equals 0 otherwise. For equal weights \(a_k=1/N\), also referred to as quasi-Monte Carlo (QMC) weights, this is the classical local discrepancy function of the point set~\(\cP\). For later use we introduce the abbreviations \begin{equation}\label{eq:Gintro}
 G_{\cP,\cA}(\boldsymbol t)
 :=\sum_{k=1}^N a_k
 \ind_{\{\boldsymbol x_k<\boldsymbol t\}}
 \qquad \mbox{and} \qquad
 v_d(\boldsymbol t):=\prod_{j=1}^d t_j,
\end{equation}
so that $\Delta_{\cP,\cA}$ can be written as $\Delta_{\cP,\cA}(\boldsymbol t)=G_{\cP,\cA}(\boldsymbol t)-v_d(\boldsymbol t)$. 

Taking a norm of the discrepancy function leads to the concept of a discrepancy. Most commonly, $L_p$-norms are studied, but discrepancies with respect to other norms also appear in the literature and applications. 

In addition to the classical approach initiated by Roth~\cite{Roth}, which investigates the asymptotic behavior of discrepancies for fixed dimension $d$ as the number $N$ of points tends to infinity (see, e.g., \cite{BC,DT97,kuinie}), the focus has increasingly shifted over the last two decades to the information-based complexity (IBC) perspective, where interest lies in the behavior of the inverse discrepancy with respect to the dimension and a prescribed relative error (see, e.g., \cite{NW10}). More precisely, the inverse discrepancy for dimension $d$ and error threshold $\varepsilon \in (0,1)$ is the minimal number of points required for the discrepancy to fall below an $\varepsilon$-fraction of the so-called initial discrepancy -- that is, the discrepancy of the empty set. From Wo\'{z}niakowski’s work \cite{Wo99} on $L_2$-discrepancy and the seminal work of Heinrich, Novak, Wasilkowski and Wo\'{z}niakowski~\cite{hnww} to the recent contributions by Novak and the second author \cite{NP25,NP26} and by the first author and Sevim~\cite{DS26}, we now have a very clear picture of the overall situation regarding $L_p$-discrepancies for nonnegative weights $a_k \ge 0$: While the $L_p$-discrepancy for finite $p\ge 1$ suffers from the curse of dimensionality, that is, its inverse increases exponentially fast with the dimension $d$, the $L_\infty$-discrepancy is polynomially tractable, that is, its inverse grows at most polynomially in $d$ and $\varepsilon^{-1}$ (see the overview \cite{NP26} and the references therein).

In addition to $L_p$-discrepancies, other notions of discrepancy have also been investigated with regard to the curse of dimensionality and tractability. Examples include the BMO seminorm in \cite{P22} and Orlicz norms in \cite{DickEtAl2020}. This paper complements the recent developments by considering the $L \log L$-Luxemburg norm, thereby adding the $L \log L$
case to the overall picture regarding the curse of dimensionality and the tractability of discrepancies with nonnegative weights. Our main result, Theorem~\ref{thm:main}, shows that, for every fixed \(\varepsilon\in(0,1)\), the number of nodes required to reduce the \(L\log L\)-Luxemburg discrepancy below an \(\varepsilon\)-fraction of its initial value grows exponentially with \(d\), even when arbitrary nonnegative weights are allowed.  

\section{\texorpdfstring{$L \log L$}{L log L}-Luxemburg norm discrepancy -- the main result}

Prominent examples of norms for discrepancies beyond the $L_p$-norm are Orlicz norms. These provide a scale of norms between and beyond the classical $L_p$-norms and have been used to study discrepancy from the classical as well as the IBC point of view; see, for example, \cite{BLPV,BM,DHMP,DickEtAl2020} and the general introduction to Orlicz spaces in \cite{RaoRen1991}.  In this paper we consider the Young function
\begin{equation}\label{eq:Phi}
 \Phi(s)=s\log(1+s) \qquad \mbox{on } \mathbb{R}_0^+,
\end{equation}
and the associated Luxemburg norm on \([0,1]^d\),
\begin{equation}\label{eq:Luxemburg}
 \norm{f}_{\Phi}
 :=\inf\left\{K>0:
 \int_{[0,1]^d}\Phi\!\left(\frac{|f(\boldsymbol t)|}{K}\right)
 \dd\boldsymbol t\leq1\right\}.
\end{equation}
This is the usual \(L\log L\)-Luxemburg norm corresponding to the particular normalization in~\eqref{eq:Phi}. As Lemma~\ref{lem:L1-embedding} and Remark~\ref{re:1} show, for any measurable $f$ on $[0,1]^d$ we have $\|f\|_{L_1} \le \kappa \|f\|_{\Phi}$ with $\kappa>0$ (see Lemma~\ref{lem:L1-embedding}) and $\|f\|_{\Phi} \le C_p \|f\|_{L_p}$ for some $C_p>0$, for all $p>1$. So $\|f\|_{\Phi}$ lies in some sense between the $L_1$-norm and the $L_p$-norms with $p>1$. 

For \(N\in\N\), $N$-element sets $\cP$ in $[0,1]^d$ and corresponding real weights $\cA=\{a_1,\ldots,a_N\}$ we define the \(L\log L\)-Luxemburg discrepancy by
$$D_N^{\Phi}(\cP,\cA):=\|\Delta_{\cP,\cA}\|_{\Phi}$$
and the $N$-th minimal \(L\log L\)-Luxemburg discrepancy for nonnegative weights 
\[
 \disc_{\Phi,+}(N,d)
 :=\inf_{\substack{\boldsymbol x_1,\ldots,\boldsymbol x_N\in[0,1]^d\\
                       a_1,\ldots,a_N\ge 0}} D_N^{\Phi}(\cP,\cA),
\]
and for QMC weights 
\[
\disc_{\Phi,\qmc}(N,d) :=\inf_{\boldsymbol x_1,\ldots,\boldsymbol x_N\in[0,1]^d } D_N^{\Phi}(\cP,\{1/N,\ldots,1/N\}).
\]
Furthermore, let $\disc_{\Phi}(0,d):=\norm{v_d}_{\Phi}$ be the initial discrepancy, that is, the discrepancy of the empty set. We will see in Proposition~\ref{prop:initial} that the initial discrepancy is of order \(d2^{-d}\); the subsequent complexity result concerns the number of nodes required to reduce this quantity by a fixed relative factor. 

\begin{remark}\label{re:duality}\rm
To place the definition of the \(L\log L\)-Luxemburg discrepancy in a broader context, we briefly explain its connection to numerical integration. The associate norm of the Luxemburg norm $\|\cdot\|_{\Phi}$ is defined as 
\begin{equation}\label{def:assoc:norm}
\norm{g}_{\Phi,*} :=\sup_{\norm{f}_{\Phi} \le1} \int_{[0,1]^d} |f(\bst)g(\bst)|\dd\bst.
\end{equation}     
Let $\cF_d$ be the anchored Sobolev--Orlicz space consisting of all functions $h$ on $[0,1]^d$ for which there exists a measurable function $g$ with $\|g\|_{\Phi,\ast}<\infty$ such that
\[
h(\bsx)=\int_{[\bszero,\bsx]}g(\bst)\dd\bst.
\]
Write $Dh:=g$, which is uniquely determined almost everywhere, and equip this space with the norm $\|h\|_{\cF_d}:=\|Dh\|_{\Phi,\ast}$.

For integration in $\cF_d$ let $I(f) = \int_{[0,1]^d} f(\bst) \dd \bst$, and use $\cA$-weighted linear rules $Q_{\cP,\cA}(f)=\sum_{k=1}^N a_k f(\bsx_k)$, where $\cP=\{\bsx_1,\ldots,\bsx_N\}$ and $\cA=\{a_1,\ldots,a_N\}$ are as before. The worst-case error of such rules $Q_{\cP,\cA}$ for integration in $\cF_d$ is defined as $$e(Q_{\cP,\cA};\cF_d):=\sup_{\substack{f \in \cF_d\\ \|f\|_{\cF_d} \le 1}}|I(f)-Q_{\cP,\cA}(f)|.$$ Then we have the ``discrepancy-integration duality'' 
\begin{equation}\label{disc:int:dual}
  e(Q_{\cP,\cA};\cF_d)=D_N^{\Phi}(\overline{\cP},\cA),
\end{equation}
where $\overline{\cP}:=\{\boldsymbol{1}-\bsx : \bsx \in \cP\}$. For the initial error $e(0;\cF_d):= \sup_{\|f\|_{\cF_d}\le 1} |I(f)|$ we have $e(0;\cF_d)=\|v_d\|_{\Phi}=\disc_{\Phi}(0,d)$. The proof of the discrepancy-integration duality is given in the Appendix.
\end{remark}

We now return to the discrepancy problem. The normalized information complexity for $\varepsilon \in (0,1)$ and $d \in \mathbb{N}$ is
\begin{equation}\label{eq:complexity}
 N_{\Phi,\bullet}(\varepsilon,d)
 :=\min\left\{N\in \N:
 \disc_{\Phi,\bullet}(N,d)\leq\varepsilon \disc_{\Phi}(0,d)\right\},
\end{equation}
where $\bullet \in \{+,\qmc\}$. Since QMC weights are a subclass of nonnegative weights,
\[
N_{\Phi,\mathrm{qmc}}(\varepsilon,d) \geq N_{\Phi,+}(\varepsilon,d).
\]
Note that these quantities are well-defined, since there exist a finite $C_\infty>0$ such that $\|\cdot\|_{\Phi} \le C_{\infty} \|\cdot\|_{L_\infty}$ and QMC sequences of point sets of size $N \in \N$ with $L_\infty$-discrepancy tending to zero for increasing $N$.

The $L \log L$-Luxemburg with nonnegative weights or QMC weights, respectively, is said to suffer from the curse of dimensionality, if there are $C,\gamma>0$ and $\varepsilon_0 \in (0,1)$ such that $N_{\Phi,\bullet}(\varepsilon,d) \ge C(1+\gamma)^d$ for all $\varepsilon \in (0,\varepsilon_0)$ and for infinitely many $d \in \mathbb{N}$.  

We now state the main result. Its proof is given in Section~\ref{sec:proof}. 

\begin{theorem}\label{thm:main}
For every \(\varepsilon\in(0,1)\), there exist constants
\(c_{\varepsilon}>0\) and \(d_0=d_0(\varepsilon)\in\N\) such that
\begin{equation}\label{eq:main-finite}
 N_{\Phi,\mathrm{qmc}}(\varepsilon,d)
 \geq N_{\Phi,+}(\varepsilon,d)
 \geq \frac{c_{\varepsilon}}{d}
 \left(\frac{3+2 \sqrt{3}}{6}\right)^d
\end{equation}
for every \(d\geq d_0\).  In particular,
\begin{equation*}
 \liminf_{d\to\infty}
 N_{\Phi,+}(\varepsilon,d)^{1/d}
 \geq \frac{3+2 \sqrt{3}}{6}=1.07735\ldots.
\end{equation*}
Consequently (note that the factor $1/d$ in \eqref{eq:main-finite} is absorbed), the \(L\log L\)-Luxemburg discrepancy with nonnegative weights, and hence also the classical QMC weight discrepancy, suffers from the curse of dimensionality.
\end{theorem}

The nonnegativity assumption is essential to the proof in Section~\ref{sec:proof}: it permits a positive-mixture representation of the normalized counting function and the subadditivity estimate used in Lemma~\ref{lem:mixture}. The corresponding problem for signed weights remains open.

The following consequence concerns numerical integration. Let $$e^+(N,d) :=\inf_{\substack{\bsx_1,\ldots,\bsx_N\in[0,1]^d\\a_1,\ldots,a_N\geq0}} e(Q_{\cP,\cA};\cF_d)$$ denote the minimal worst-case error of a linear rule $Q_{\cP,\cA}$ with nonnegative weights in the space $\cF_d$ from Remark~\ref{re:duality} and consider the normalized information complexity
$$N_{{\rm int},+}(\varepsilon,d) :=\min\{N\in\N: e^+(N,d)\le \varepsilon \, e(0; \cF_d)\} \qquad \mbox{for } \varepsilon \in (0,1),\ d \in \mathbb{N}.$$ The discrepancy-integration duality \eqref{disc:int:dual} implies $e^+(N,d)={\rm disc}_{\Phi,+}(N,d)$, $e(0;\cF_d)=\disc_{\Phi}(0,d)$ and $N_{{\rm int},+}(\varepsilon,d)=N_{\Phi,+}(\varepsilon,d)$. Thus, Theorem~\ref{thm:main} yields:  

\begin{corollary}
Integration in the space $\cF_d$ by means of linear rules with nonnegative weights suffers from the curse of dimensionality.
\end{corollary}

\section{Luxemburg-norm preliminaries and the initial discrepancy}

For a measurable function \(f\) on \([0,1]^d\), write for short $\rho_\Phi(f):=\int_{[0,1]^d}\Phi(|f(\boldsymbol t)|)\dd\boldsymbol t$.

\begin{lemma}\label{lem:modular}
Let \(f\) be a bounded measurable function not a.e. zero.  Then the function
\[
 K \mapsto \rho_\Phi(f/K)\qquad \mbox{for } K>0,
\]
is continuous and strictly decreasing, tends to \(+\infty\) as \(K \downarrow0\), and tends to \(0\) as \(K\to\infty\).  Consequently,
\begin{equation*}
 \rho_\Phi\!\left(\frac{f}{\norm{f}_{\Phi}}\right)=1.
\end{equation*}
Moreover, if \(K \geq\norm{f}_{\Phi}\), then $\rho_\Phi(f/K)\leq 1$.
\end{lemma}

\begin{proof}
The function \(\Phi\) is continuous and strictly increasing on \([0,\infty)\).  Hence, for every point at which \(|f|>0\), the quantity
\(\Phi(|f|/K)\) is strictly decreasing in \(K\).  Since \(f\) is bounded, dominated convergence shows that \(K\mapsto\rho_\Phi(f/K)\) is continuous on every compact subinterval of \((0,\infty)\).  The limit as \(K\to\infty\) is zero by dominated convergence.  Since \(f\) is not zero almost everywhere, there are numbers \(c>0\) and \(m>0\) such that the set \(\{|f|\geq c\}\) has measure at least \(m\).  Therefore
\[
 \rho_\Phi(f/K)\geq m\Phi(c/K)\rightarrow\infty
 \qquad \mbox{as } K\downarrow 0.
\]
Thus there is a unique \(K_0>0\) for which \(\rho_\Phi(f/K_0)=1\).  By the definition of the Luxemburg norm, this \(K_0\) equals \(\norm{f}_{\Phi}\).  The last assertion follows from monotonicity.
\end{proof}

Let \(\kappa>0\) be the unique solution of $\Phi(\kappa)=\kappa\log(1+\kappa)=1$.
Numerically, \(\kappa=1.239977887\ldots\).

\begin{lemma}[Embedding into \(L_1\)]\label{lem:L1-embedding}
For every measurable \(f\) on \([0,1]^d\) we have $\norm{f}_{L_1}\leq\kappa\norm{f}_{\Phi}$.
\end{lemma}

\begin{proof}
The assertion is trivial if \(f=0\) almost everywhere or \(\norm{f}_{\Phi}=\infty\). Otherwise, let \(K>\norm{f}_{\Phi}\). By the definition of the infimum, there exists \(K_0<K\) such that \(\rho_\Phi(f/K_0)\leq1\); hence monotonicity gives
\[
 \int_{[0,1]^d}\Phi\!\left(\frac{|f(\bst)|}{K}\right) \dd \bst\leq1.
\]
Since the unit cube has Lebesgue measure one and \(\Phi\) is convex, Jensen's inequality gives
\[
 \Phi\!\left(\frac{\norm{f}_{L_1}}{K}\right)
 \leq\int_{[0,1]^d}\Phi\!\left(\frac{|f(\bst)|}{K}\right) \dd \bst
 \leq 1=\Phi(\kappa).
\]
The function \(\Phi\) is increasing, so \(\norm{f}_{L_1}\leq\kappa K\).  Letting \(K\downarrow\norm{f}_{\Phi}\) proves Lemma~\ref{lem:L1-embedding}.
\end{proof}

\begin{remark}\label{re:1}\rm
For every $p>1$ there exists a $C_p>0$ with the following property: for every measurable \(f\) on \([0,1]^d\) we have $\norm{f}_{\Phi}\leq C_p \norm{f}_{L_p}$. The proof is given in the Appendix.
\end{remark}

We now analyze the initial discrepancy. Note that $\norm{v_d}_{L_1}=\int_{[0,1]^d}v_d(\boldsymbol t)\dd\boldsymbol t=2^{-d}$. Put
\begin{equation}\label{eq:md-ad}
\alpha_d:=2^d \norm{v_d}_{\Phi}.
\end{equation}
Define the volume-biased probability measure \(\nu_d\) by
\begin{equation}\label{eq:nud}
 \dd\nu_d(\boldsymbol t)
 :=2^dv_d(\boldsymbol t)\dd\boldsymbol t.
\end{equation}
This is the product measure
\begin{equation}\label{eq:productmeasure}
 \nu_d=\mu^{\otimes d}
 \qquad \mbox{where } \dd\mu(t)=2t\dd t.
\end{equation}

\begin{proposition}[Initial discrepancy]\label{prop:initial}
The initial $L \log L$-Luxemburg discrepancy is
\begin{equation*}
\disc_{\Phi}(0,d)= \norm{v_d}_{\Phi} =\left(c_0+o(1)\right) \frac{d}{2^d},\qquad \mbox{where } c_0:=\log2-\frac12.
\end{equation*}
\end{proposition}

\begin{proof}
We prove the equivalent assertion 
\begin{equation}\label{eq:ad-asymptotic}
 \lim_{d\to\infty}\frac{\alpha_d}{d}=c_0.
\end{equation}

By Lemma~\ref{lem:modular} with $f(\bst)=v_d(\bst)$, the defining modular equals one at \(K =\norm{v_d}_{\Phi}\), i.e., 
\[
 1=\int_{[0,1]^d}
 \frac{v_d(\boldsymbol t)}{\norm{v_d}_{\Phi}}
 \log\!\left(1+\frac{v_d(\boldsymbol t)}{\norm{v_d}_{\Phi}}\right)
 \dd\boldsymbol t.
\]
Using \eqref{eq:nud} and \(\norm{v_d}_{\Phi}=\alpha_d 2^{-d}\), we obtain
\begin{equation}\label{eq:ad-equation}
 \alpha_d
 =\E_{\nu_d}\left[
 \log\!\left(1+\frac{2^d \, v_d(\boldsymbol T)}{\alpha_d}\right)\right],
\end{equation}
where \(\boldsymbol T=(T_1,\ldots,T_d)\) has distribution \(\nu_d\).

We first establish crude bounds on \(\alpha_d\).  By Lemma~\ref{lem:L1-embedding}, $2^{-d}=\norm{v_d}_{L_1}\leq\kappa \, \norm{v_d}_{\Phi}$, and hence
\begin{equation}\label{eq:ad-lower}
 \alpha_d\geq\kappa^{-1}.
\end{equation}
On the other hand, the concavity of \(x\mapsto\log(1+x)\), followed by Jensen's inequality in \eqref{eq:ad-equation}, gives
\begin{align*}
 \alpha_d
 &\leq\log\!\left(1+
 \frac{2^d \, \E_{\nu_d}[v_d(\boldsymbol T)]}{\alpha_d}\right).
\end{align*}
A direct product calculation yields
\[
 \E_{\nu_d}[v_d(\boldsymbol T)]
 =2^d\int_{[0,1]^d}v_d(\boldsymbol t)^2\dd\boldsymbol t
 =\left(\frac23\right)^d.
\]
Therefore,
\begin{equation}\label{eq:ad-upper}
 \alpha_d
 \leq\log\!\left(1+\frac{(4/3)^d}{\alpha_d}\right)
 \leq\log\!\left(1+\kappa \left(\tfrac43\right)^d\right)
 \leq d\log\tfrac43+\log(1+\kappa).
\end{equation}
It follows from \eqref{eq:ad-lower} and \eqref{eq:ad-upper} that
\begin{equation}\label{eq:logad}
\lim_{d \rightarrow \infty} \frac{\log \alpha_d}{d} = 0.
\end{equation}

Under \(\mu\) from \eqref{eq:productmeasure}, the random variable \(Y:=-\log T\) is exponentially distributed with rate two.  Indeed, for \(y\geq0\),
\[
 \mu(Y>y)=\mu(T< {\rm e}^{-y})=\int_0^{{\rm e}^{-y}}2t\dd t={\rm e}^{-2y}.
\]
Consequently, if
\begin{equation}\label{eq:Sd}
 S_d:=-\log v_d(\boldsymbol T)=\sum_{j=1}^d(-\log T_j),
\end{equation}
then
\begin{equation}\label{eq:Sd-moments}
 \E_{\nu_d}[S_d]=\frac d2,
 \qquad
 \operatorname{Var}_{\nu_d}[S_d]=\frac d4.
\end{equation}
We rewrite Eq.~\eqref{eq:ad-equation} as
\begin{equation*}
 \alpha_d  =\E_{\nu_d}\left[\ell\!\left(d\log2-S_d-\log \alpha_d\right)\right]
 \qquad \mbox{where }  \ell(x):=\log(1+{\rm e}^x).
\end{equation*}
Dividing by \(d\) and using the elementary estimate $x_+\leq\ell(x)\leq x_+ +\log2$ for every \(x\in\R\), where $x_+:=\max\{x,0\}$, we obtain
\begin{equation}\label{eq:ad-positive-part}
 \left|
 \frac{\alpha_d}{d}
 -\E_{\nu_d}\left[\left(
 \log2-\frac{S_d}{d}-\frac{\log \alpha_d}{d}
 \right)_+\right]
 \right|
 \leq\frac{\log2}{d}.
\end{equation}
The map \(x\mapsto x_+\) is 1-Lipschitz.  Hence, by \eqref{eq:Sd-moments} and Cauchy-Schwarz,
\begin{align*}
 \left|
 \E_{\nu_d}\left[\left(
 \log2-\frac{S_d}{d}-\frac{\log \alpha_d}{d}
 \right)_+\right]-c_0
 \right|
 &\leq
 \E_{\nu_d}\left[\left|\frac{S_d}{d}-\frac12\right|\right]
 +\frac{|\log \alpha_d|}{d}\\
 &\leq
 \sqrt{\operatorname{Var}_{\nu_d}\left[\frac{S_d}{d}\right]}
 +\frac{|\log \alpha_d|}{d}
 =\frac{1}{2\sqrt d}+\frac{|\log \alpha_d|}{d}.
\end{align*}
The right-hand side tends to zero by \eqref{eq:logad}.  Combining this with \eqref{eq:ad-positive-part} proves \eqref{eq:ad-asymptotic}, and the final result follows from \(\norm{v_d}_{\Phi}=\alpha_d2^{-d}\).
\end{proof}

\section{The proof of Theorem~\ref{thm:main}}\label{sec:proof}

Fix nodes $\cP$ and nonnegative weights $\cA$, and let \(G_{\cP,\cA}\) be given by \eqref{eq:Gintro}.  On the set on which \(v_d>0\), define
\begin{equation}\label{eq:Z}
 Z(\boldsymbol t):=\frac{G_{\cP,\cA}(\boldsymbol t)}{v_d(\boldsymbol t)}
 \qquad \mbox{ and }\qquad  X(\boldsymbol t):=|Z(\boldsymbol t)-1|.
\end{equation}
On the union of the coordinate hyperplanes, where \(v_d=0\), set \(Z:=0\) and \(X:=1\). This exceptional set has \(\nu_d\)-measure zero. Since $|\Delta_{\cP,\cA}|=|G_{\cP,\cA}-v_d|=v_dX$ almost everywhere, the volume-biased measure $\nu_d$ is adapted to the relative discrepancy \(X\).

The proof of the main result has two principal steps. First, small normalized discrepancy implies a dimension-independent lower bound for \(\mathbb E_{\nu_d}[Z^{1/2}]\) on volume-typical boxes (Lemma~\ref{lem:localization}). Second, each anchored-box atom loses a fixed factor per coordinate in this fractional moment, while a positive mixture of \(N\) atoms gains at most a factor \(N^{1/2}\) (Lemmas~\ref{lem:one-dimensional} and \ref{lem:mixture}). Comparing the bounds yields the exponential lower bound on \(N\). The proof is based on methods from \cite{DS26}, but requires a higher level of technical detail due to the more complicated norm.

\begin{lemma}\label{lem:modular-estimate}
Suppose that
\begin{equation}\label{eq:small-norm}
 \norm{\Delta_{\cP,\cA}}_{\Phi}\leq\varepsilon \norm{v_d}_{\Phi} \qquad \mbox{for some } \varepsilon \in (0,1).
\end{equation}
Then
\begin{equation}\label{eq:modular-X}
 \E_{\nu_d}\left[
 X\log\!\left(1+\frac{v_d}{\varepsilon \, \norm{v_d}_{\Phi}}X\right)
 \right]
 \leq\varepsilon \alpha_d \qquad \mbox{and} \qquad \E_{\nu_d}[X]\leq\kappa\varepsilon \alpha_d.
\end{equation}
Consequently,
\begin{equation}\label{eq:EZ-upper}
 \E_{\nu_d}[Z]\leq1+\kappa\varepsilon \alpha_d.
\end{equation}
\end{lemma}

\begin{proof}
The case $\Delta_{\cP,\cA}=0$ a.e. is trivial. If $\Delta_{\cP,\cA}$ is not a.e. zero, then Lemma~\ref{lem:modular} and \eqref{eq:small-norm} give
\[
 \int_{[0,1]^d}
 \Phi\!\left(\frac{|\Delta_{\cP,\cA}(\bst)|}{\varepsilon \norm{v_d}_{\Phi}}\right)
 \dd\boldsymbol t\leq1.
\]
Since \(|\Delta_{\cP,\cA}|=v_dX\), the definition of \(\Phi\), the identity
\(\norm{v_d}_{\Phi}=\alpha_d2^{-d}\), and the change of measure \eqref{eq:nud} give
\begin{align*}
 1
 &\geq\int_{[0,1]^d}
 \frac{v_d(\bst)X(\bst)}{\varepsilon \norm{v_d}_{\Phi}}
 \log\!\left(1+\frac{v_d(\bst)X(\bst)}{\varepsilon \norm{v_d}_{\Phi}}\right)
 \dd\boldsymbol t=\frac{1}{\varepsilon \alpha_d}
 \E_{\nu_d}\left[
 X\log\!\left(1+\frac{v_d}{\varepsilon \norm{v_d}_{\Phi}}X\right)
 \right].
\end{align*}
This implies the first estimate in \eqref{eq:modular-X}. Furthermore, Lemma~\ref{lem:L1-embedding} and \eqref{eq:small-norm} imply $\norm{\Delta_{\cP,\cA}}_{L_1}
 \leq\kappa\varepsilon \norm{v_d}_{\Phi}$. Using \(\dd\nu_d=(v_d/2^{-d})\dd\boldsymbol t\), we obtain
\[
 \E_{\nu_d}[X]
 =\frac{1}{2^{-d}}\int_{[0,1]^d}v_d(\bst)X(\bst)\dd\boldsymbol t
 =\frac{\norm{\Delta_{\cP,\cA}}_{L_1}}{2^{-d}}
 \leq\kappa\varepsilon\frac{\norm{v_d}_{\Phi}}{2^{-d}}
 =\kappa\varepsilon \alpha_d,
\]
which is the second estimate in \eqref{eq:modular-X}. Finally, \(Z\leq1+|Z-1|=1+X\), which proves~\eqref{eq:EZ-upper}.
\end{proof}

Next, we localize to boxes whose volume is typical under \(\nu_d\).  Recall \(c_0=\log2-\frac12\).

\begin{lemma}\label{lem:localization}
Fix \(\varepsilon\in(0,1)\).  Choose numbers \(\delta\) and \(\tau\) satisfying
\begin{equation}\label{eq:delta-tau}
 0<\delta<c_0(1-\varepsilon) \qquad \mbox{ and } \qquad 
 0<\tau<1-\frac{\varepsilon c_0}{c_0-\delta}.
\end{equation}
Define
\begin{equation*}
 I_\delta:=2\delta-\log(1+2\delta)>0
\qquad \mbox{ and } \qquad 
 L_d:=\log\!\left(
 1+\frac{\tau {\rm e}^{(c_0-\delta)d}}{\varepsilon \alpha_d}
 \right).
\end{equation*}
If \eqref{eq:small-norm} holds, then
\begin{equation}\label{eq:moment-lower-explicit}
 \E_{\nu_d}[Z^{1/2}]
 \geq
 1-{\rm e}^{-I_\delta d}-\tau-\frac{\varepsilon \alpha_d}{L_d}.
\end{equation}
In particular, there exist \(b=b(\varepsilon,\delta,\tau)>0\) and
\(d_1=d_1(\varepsilon,\delta,\tau)\) such that
\begin{equation}\label{eq:moment-lower}
 \E_{\nu_d}[Z^{1/2}]\geq b \qquad \mbox{for all } d\geq d_1.
\end{equation}
The constants \(b\) and \(d_1\) are independent of the particular nodes and weights.
\end{lemma}

\begin{proof}
Let \(S_d\) be as in \eqref{eq:Sd} and define $A_{d,\delta} :=\{S_d\leq (\frac12+\delta)d\}$. We first estimate the probability of its complement.  If \(Y\) is exponential with rate two, then, for \(\lambda \in (0,2)\),
\[
 \E[{\rm e}^{\lambda Y}]=\frac{2}{2-\lambda}.
\]
Therefore, by the exponential Markov inequality, 
\[
 \nu_d\left(S_d\geq (\tfrac12+\delta)d\right)
 \leq {\rm e}^{-\lambda (1/2+\delta) d}\, \left(\frac{2}{2-\lambda}\right)^d.
\]
Choosing $\lambda=2-(1/2+\delta)^{-1}=\frac{4\delta}{1+2\delta}\in(0,2)$ yields
\begin{align*}
 \nu_d(A_{d,\delta}^c)
 &\leq {\rm e}^{-d (2\delta-\log(1+2\delta))} ={\rm e}^{-I_\delta d}.
\end{align*}
The number \(I_\delta\) is positive because
\(x-\log(1+x)>0\) for every \(x>0\).

On \(A_{d,\delta}\),
\[
 v_d(\boldsymbol T)={\rm e}^{-S_d}
 \geq {\rm e}^{-(1/2+\delta)d}.
\]
Since \(\norm{v_d}_{\Phi}=\alpha_d2^{-d}\), it follows that
\begin{equation}\label{eq:volume-ratio-lower}
 \frac{v_d(\boldsymbol T)}{\varepsilon \norm{v_d}_{\Phi}}
 \geq\frac{{\rm e}^{(\log2-1/2-\delta)d}}{\varepsilon \alpha_d}
 =\frac{{\rm e}^{(c_0-\delta)d}}{\varepsilon \alpha_d}
 \qquad\text{on }A_{d,\delta}.
\end{equation}
For every \(x\geq0\), \eqref{eq:volume-ratio-lower} implies, pointwise on
\(A_{d,\delta}\),
\begin{equation}\label{eq:pointwise-localization}
 x
 \leq\tau+
 \frac{x\log\!\left(1+\frac{v_d(\boldsymbol T)}{\varepsilon \norm{v_d}_{\Phi}}x\right)}{L_d}.
\end{equation}
Indeed, if \(x<\tau\), then the right-hand side is at least \(\tau>x\).  If \(x\geq\tau\), then
\[
 \log\!\left(1+\frac{v_d(\boldsymbol T)}{\varepsilon \norm{v_d}_{\Phi}}x\right)
 \geq\log\!\left(1+\frac{\tau {\rm e}^{(c_0-\delta)d}}{\varepsilon \alpha_d}\right)
 =L_d,
\]
so the second term on the right-hand side of \eqref{eq:pointwise-localization} is at least \(x\).

Apply \eqref{eq:pointwise-localization} with \(x=X\), multiply by
\(\ind_{A_{d,\delta}}\), take expectations, and use
\eqref{eq:modular-X}.  We obtain
\begin{equation}\label{eq:localized-X}
\E_{\nu_d}[X\ind_{A_{d,\delta}}] \leq\tau+\frac{\varepsilon \alpha_d}{L_d}.
\end{equation}

For \(z\geq0\) we have $z^{1/2}\geq1-|z-1|$. Indeed, for \(0\leq z\leq1\), the right-hand side equals
\(z \le z^{1/2}\); for $z \ge 1$ it equals \(2-z\leq1\leq z^{1/2}\).  Applying our estimate to \(Z\) on \(A_{d,\delta}\) gives
\begin{align*}
\E_{\nu_d}[Z^{1/2}] &\geq\E_{\nu_d}[Z^{1/2}\ind_{A_{d,\delta}}] \geq\nu_d(A_{d,\delta}) -\E_{\nu_d}[X\ind_{A_{d,\delta}}].
\end{align*}
The probability estimate and \eqref{eq:localized-X} prove
\eqref{eq:moment-lower-explicit}.

It remains to show that the right-hand side of
\eqref{eq:moment-lower-explicit} is eventually positive.  By
\eqref{eq:ad-asymptotic}, \(\lim_{d \rightarrow \infty} \alpha_d/d = c_0\) and by \eqref{eq:logad}, \(\lim_{d \rightarrow \infty} \log \alpha_d/d = 0\).  Writing
\[
 L_d=\log(1+{\rm e}^{q_d}) \qquad \mbox{ and } \qquad 
 q_d=(c_0-\delta)d-\log \alpha_d+\log(\tau/\varepsilon),
\]
we have \(\lim_{d \rightarrow \infty} q_d/d = c_0-\delta>0\), and hence $\lim_{d \rightarrow \infty} L_d/d = c_0-\delta$. Therefore,
\begin{equation*}
 \lim_{d \rightarrow \infty} \frac{\varepsilon \alpha_d}{L_d} = \frac{\varepsilon c_0}{c_0-\delta}.
\end{equation*}
By \eqref{eq:delta-tau},
\[
 b_\infty
 :=1-\tau-\frac{\varepsilon c_0}{c_0-\delta}>0.
\]
The right-hand side of \eqref{eq:moment-lower-explicit} converges to
\(b_\infty\).  Thus it is at least, for example,
\(b=b_\infty/2>0\) for all sufficiently large \(d\), proving
\eqref{eq:moment-lower}.
\end{proof}

\begin{lemma}\label{lem:one-dimensional}
Let \(\boldsymbol x\in[0,1]^d\) and define
\begin{equation}\label{eq:Yx}
 Y_{\boldsymbol x}(\boldsymbol t)
 :=\frac{\ind_{\{\boldsymbol x<\boldsymbol t\}}}{v_d(\boldsymbol t)}
\end{equation}
when \(v_d(\boldsymbol t)>0\), and set it equal to zero otherwise. Set $R:=\sqrt{4 \sqrt{3}-6}$. Then
\begin{equation}\label{eq:tensor-contraction}
\E_{\nu_d}[Y_{\boldsymbol x}^{1/2}] \leq R^d \left(\E_{\nu_d}[Y_{\boldsymbol x}]\right)^{1/2}.
\end{equation}
\end{lemma}

\begin{proof}
Because of the product structure it suffices to prove the result in the one-dimensional case for $x \in [0,1)$. (Note that if some $x_j=1$, both sides of \eqref{eq:tensor-contraction} vanish.) With respect to \(\mu\) from \eqref{eq:productmeasure}, direct integration gives
\begin{align*}
 \E_\mu[Y_x]
 &=\int_x^1\frac1t\,2t\dd t=2(1-x),\ \mbox{and}
 \\
 \E_\mu [Y_x^{1/2}]
 &=\int_x^1\frac1{t^{1/2}}\,2t\dd t
 =\frac{4}{3}\left(1-x^{3/2}\right).
\end{align*}
Thus for $x \in [0,1)$, $$\frac{\E_\mu [Y_x^{1/2}]}{(\E_\mu[Y_x])^{1/2}}=\frac{2 \sqrt{2}}{3} \frac{1-x^{3/2}}{(1-x)^{1/2}} \le R,$$ where the last inequality follows from a one-variable maximization (equality holds for $x=\tfrac12(2-\sqrt{3})$).
\end{proof}

\begin{lemma}\label{lem:mixture}
Let \(G_{\cP,\cA}\) be given by \eqref{eq:Gintro}, with weights \(a_k\geq0\), and let \(Z=G_{\cP,\cA}/v_d\) be as in \eqref{eq:Z}. Then
\begin{equation*}
\E_{\nu_d}[Z^{1/2}] \leq R^d N^{1/2} \left(\E_{\nu_d}[Z]\right)^{1/2}.
\end{equation*}
\end{lemma}

\begin{proof}
For \(\theta \in (0,1)\) and \(z_1,\ldots,z_N\geq0\) we will use the estimates
\begin{equation}\label{eq:subadditivity}
 \left(\sum_{k=1}^N z_k\right)^\theta
 \leq\sum_{k=1}^N z_k^\theta \le N^{1-\theta} \left(\sum_{k=1}^N z_k\right)^\theta.
\end{equation}
Here the left inequality is well known (see, e.g., \cite[Eq.~(4.26)]{LP26} for a proof) and the right inequality follows from Jensen's inequality for the concave function \(x\mapsto x^\theta\), applied to the uniform probability measure on \(\{1,\ldots,N\}\).

By \eqref{eq:Yx}, $Z=\sum_{k=1}^Na_kY_{\boldsymbol x_k}$ $\nu_d$-almost surely and all summands are nonnegative. Using \eqref{eq:subadditivity} with $\theta=1/2$ and \eqref{eq:tensor-contraction}, we get
\begin{align*}
 \E_{\nu_d}[Z^{1/2}] \leq\sum_{k=1}^N a_k^{1/2}
 \E_{\nu_d}[Y_{\boldsymbol x_k}^{1/2}]\leq R^d  \sum_{k=1}^N \left(a_k\E_{\nu_d}[Y_{\boldsymbol x_k}]\right)^{1/2} \le R^d N^{1/2}
 \left(\E_{\nu_d}[Z]\right)^{1/2},
\end{align*}
as claimed.
\end{proof}

\begin{proof}[Proof of Theorem~\ref{thm:main}]
Fix \(\varepsilon\in(0,1)\) and choose
\(\delta\) and \(\tau\) satisfying \eqref{eq:delta-tau}.  Let an $N$-point set $\cP$ and corresponding nonnegative weights $\cA$ satisfy $\norm{\Delta_{\cP,\cA}}_{\Phi}\leq\varepsilon \norm{v_d}_{\Phi}$.

By Lemma~\ref{lem:localization}, there are constants \(b>0\) and \(d_1\) such that
\begin{equation*}
\E_{\nu_d}[Z^{1/2}]\geq b \qquad \mbox{for every } d\geq d_1,
\end{equation*}
and by Lemma~\ref{lem:mixture} and \eqref{eq:EZ-upper},
\begin{equation*}
 \E_{\nu_d}[Z^{1/2}] \leq R^d N^{1/2}
 \left(1+\kappa\varepsilon \alpha_d\right)^{1/2}.
\end{equation*}
Combining these two estimates, we obtain the explicit bound
\begin{equation}\label{eq:explicit-rule-bound}
 N
 \geq b^2
 \left(1+\kappa\varepsilon \alpha_d\right)^{-1}
 R^{-2d}.
\end{equation}

By \eqref{eq:ad-asymptotic}, \(\alpha_d/d\to c_0\).  Therefore, after increasing \(d_1\) if necessary, we may assume that \(\alpha_d\leq2c_0d\).  Since
\(d\geq1\), we have $1+\kappa\varepsilon \alpha_d
 \leq (1+2\kappa\varepsilon c_0)d$.
Thus \eqref{eq:explicit-rule-bound} implies
\[
 N
 \geq \frac{c_{\varepsilon}}{d}
 R^{-2d},
\]
where one may take $c_{\varepsilon} :=b^2 (1+2\kappa\varepsilon c_0)^{-1}>0$. This lower bound holds for every $N$-point set $\cP$ with corresponding nonnegative weights $\cA$ whose normalized $L\log L$-Luxemburg discrepancy is at most \(\varepsilon\). 

If the infimum defining \(\disc_{\Phi,+}(N,d)\) is not attained we proceed as follows. Fix $\varepsilon\in(0,1)$, put $\varepsilon':=(1+\varepsilon)/2\in(\varepsilon,1)$, and choose $\delta$ and $\tau$ satisfying \eqref{eq:delta-tau} with $\varepsilon'$ in place of $\varepsilon$. Let $N:=N_{\Phi,+}(\varepsilon,d)$. Since $\disc_{\Phi,+}(N,d)\le\varepsilon\,\disc_{\Phi}(0,d)<\varepsilon'\disc_{\Phi}(0,d)$, there exist an $N$-point set $\cP$ and nonnegative weights $\cA$ with $\norm{\Delta_{\cP,\cA}}_{\Phi}\le\varepsilon'\norm{v_d}_{\Phi}$. The theorem follows.
\end{proof}

\section{Appendix}

\subsection*{Discrepancy-integration duality}

The complementary (convex conjugate) Young function $\Phi^{\ast}$ is defined as $$\Phi^{\ast}(s):=\sup_{t \ge 0}(s t-\Phi(t)) \qquad \mbox{for } s \ge 0.$$ Since $\Phi'$ is continuous and strictly increasing with $\Phi'(0)=0$ and $\Phi'(t) \rightarrow \infty$ as $t \rightarrow \infty$, the supremum is attained at $t$ with $s=\Phi'(t)$, i.e., at $t=(\Phi')^{-1}(s)$. Hence 
\begin{equation}\label{eq:Aast}
\Phi^{\ast}(s)=s (\Phi')^{-1}(s)-\Phi((\Phi')^{-1}(s)) 
\end{equation}
and thus 
\begin{equation}\label{eq:compAalpha}
\Phi^{\ast}(\Phi'(t))=t \Phi'(t)-\Phi(t).
\end{equation}

\begin{lemma}[Orlicz--H\"older inequality]\label{le:OrHo}
Let $f$ be bounded and measurable and $g$ measurable. We have
\begin{equation}\label{eq:orlicz-holder-associate}
        \int_{[0,1]^d} |f(\bst)g(\bst)|\dd\bst
        \le
        \norm{f}_{\Phi}\norm{g}_{\Phi,*}.
\end{equation}
The inequality is sharp. For bounded, measurable and not a.e.~zero  functions define the so-called norming function
\[
        g_f(\bst)
        :=
        \frac{
        \Phi'\left(\frac{|f(\bst)|}{\norm{f}_{\Phi}}\right)
        }{\int_{[0,1]^d}
        \frac{|f(\bss)|}{\norm{f}_{\Phi}}
        \Phi'\left(\frac{|f(\bss)|}{\norm{f}_{\Phi}}\right)
        \dd\bss
        } .
\]
Then
\[
        \norm{g_f}_{\Phi,*}=1
        \qquad \mbox{ and } \qquad
        \int_{[0,1]^d}|f(\bst)|g_f(\bst)\dd\bst
        =\norm{f}_{\Phi} ,
\]
i.e., for $f$ and $g_f$ we have equality in \eqref{eq:orlicz-holder-associate}.
\end{lemma}

\begin{proof}
For $f \not=0$ we have 
\begin{align*}
\int_{[0,1]^d}|f(\bst)g(\bst)|\dd\bst = & \|f\|_{\Phi} \int_{[0,1]^d} \frac{|f(\bst)|}{\|f\|_{\Phi}} |g(\bst)| \dd \bst \\ \le &  \|f\|_{\Phi} \sup_{\|h\|_{\Phi}\le 1} \int_{[0,1]^d} |h(\bst) g(\bst)| \dd \bst = \|f\|_{\Phi}  \norm{g}_{\Phi,*}.  
\end{align*}
This proves \eqref{eq:orlicz-holder-associate}. It remains to prove the sharpness of the bound.

On the one hand we have
\begin{align*}
\|g_f\|_{\Phi,\ast} = & \sup_{\|h\|_{\Phi} \le 1} \frac{\int_{[0,1]^d}|h(\bst)| \Phi'\left(\frac{|f(\bst)|}{\norm{f}_{\Phi}}\right) \dd \bst}{
\int_{[0,1]^d} \frac{|f(\bss)|}{\norm{f}_{\Phi}} \Phi'\left(\frac{|f(\bss)|}{\norm{f}_{\Phi}}\right) \dd\bss}\ge  \frac{\int_{[0,1]^d} \frac{|f(\bst)|}{\norm{f}_{\Phi}} \Phi'\left(\frac{|f(\bst)|}{\norm{f}_{\Phi}}\right)\dd \bst}{
\int_{[0,1]^d} \frac{|f(\bss)|}{\norm{f}_{\Phi}} \Phi'\left(\frac{|f(\bss)|}{\norm{f}_{\Phi}}\right)\dd\bss}=1.
\end{align*}
On the other hand, choose any function $h$ with $\|h\|_{\Phi}\le 1$. Then for any $K>1\ge \|h\|_{\Phi}$ we have $\rho_{\Phi}(h/K) \le 1$, and monotone convergence for $K \downarrow 1$ gives $\int_{[0,1]^d}\Phi(|h(\bst)|) \dd \bst \le 1$. Using Young's inequality $v w \le \Phi(v)+\Phi^{\ast}(w)$ with $v=|h(\bst)|$ and $w=\Phi'(|f(\bst)|/\|f\|_{\Phi})$ pointwise and integration over $[0,1]^d$ yields 
\begin{align*}
\lefteqn{\int_{[0,1]^d} |h(\bst)| \Phi'\left(\frac{|f(\bst)|}{\|f\|_{\Phi}}\right) \dd \bst}\\
& \le \int_{[0,1]^d} \Phi(|h(\bst)|) \dd \bst   +  \int_{[0,1]^d} \Phi^{\ast}\left(\Phi'\left(\frac{|f(\bst)|}{\|f\|_{\Phi}}\right)\right) \dd \bst  \\
& =  \int_{[0,1]^d} \Phi(|h(\bst)|) \dd \bst   +  \int_{[0,1]^d} \frac{|f(\bst)|}{\|f\|_{\Phi}} \Phi'\left(\frac{|f(\bst)|}{\|f\|_{\Phi}}\right) \dd \bst -  \int_{[0,1]^d} \Phi\left(\frac{|f(\bst)|}{\|f\|_{\Phi}}\right) \dd \bst,
\end{align*}
where we used \eqref{eq:compAalpha}. Since $\int_{[0,1]^d} \Phi(|h(\bst)|) \dd \bst \le 1$ and $\int_{[0,1]^d} \Phi\left(\frac{|f(\bst)|}{\|f\|_{\Phi}}\right) \dd \bst=1$ it follows that 
$$\int_{[0,1]^d} |h(\bst)| \Phi'\left(\frac{|f(\bst)|}{\|f\|_{\Phi}}\right) \dd \bst \le \int_{[0,1]^d} \frac{|f(\bst)|}{\|f\|_{\Phi}} \Phi'\left(\frac{|f(\bst)|}{\|f\|_{\Phi}}\right) \dd \bst.$$ Thus, for every function $h$ with $\|h\|_{\Phi}\le 1$ we have $$\int_{[0,1]^d} |h(\bst) g_f(\bst)| \dd \bst=\frac{\int_{[0,1]^d}|h(\bst)| \Phi'\left(\frac{|f(\bst)|}{\norm{f}_{\Phi}}\right)\dd \bst}{
\int_{[0,1]^d} \frac{|f(\bst)|}{\norm{f}_{\Phi}} \Phi'\left(\frac{|f(\bst)|}{\norm{f}_{\Phi}}\right) \dd\bst}\le 1$$ and this implies that $\|g_f\|_{\Phi,\ast} \le 1$. Together with the lower bound we deduce $\|g_f\|_{\Phi,\ast}=1$.

Finally, it is obvious that 
$$\int_{[0,1]^d}|f(\bst) g_f(\bst)|\dd\bst =  \frac{\int_{[0,1]^d}|f(\bst)| 
        \Phi'\left(\frac{|f(\bst)|}{\|f\|_{\Phi}}\right) \dd \bst
        }{\int_{[0,1]^d}
        \frac{|f(\bss)|}{\|f\|_{\Phi}}
        \Phi'\left(\frac{|f(\bss)|}{\|f\|_{\Phi}}\right)
        \dd\bss} = \|f\|_{\Phi}.$$
This finishes the proof.
\end{proof}

Now we can prove the discrepancy-integration duality stated in Remark~\ref{re:duality}.

\begin{proof}[Proof of Remark~\ref{re:duality}]
As an auxiliary function we define the upper anchored local discrepancy
\begin{equation*}
\nabla_{\cP,\cA}(\bst) :=\prod_{j=1}^d(1-t_j) -\sum_{k=1}^N a_k\ind_{\{\bst \le \bsx_k\}} \qquad \mbox{for } \bst \in [0,1]^d,
\end{equation*}
which is obviously not a.e. zero, because $\prod_{j=1}^d(1-t_j)$ is not a step function.
Let $\eta:=\lambda-\sum_{k=1}^N a_k \delta_{\bsx_k}$ where $\delta_{\bsx_k}$ is the Dirac measure at $\bsx_k$ and $\lambda$ is the Lebesgue measure on $[0,1]^d$ and let $L(f):= I(f)-Q_{\cP,\cA}(f)$. 
For $f \in \cF_d$ we have $$f(\bsx)=\int_{[0,1]^d} Df(\bst) {\bf 1}_{\{\bst < \bsx\}} \dd \bst.$$ From this identity and an application of Fubini's theorem (note that $\|D f\|_{L_1} \le \|1\|_{\Phi} \|D f\|_{\Phi,\ast}$ by Lemma~\ref{le:OrHo}, and hence $D f \in L_1$) we deduce \begin{eqnarray*}
L(f) = \int_{[0,1]^d} D f(\bst) \eta([\bst,\boldsymbol{1}]) \dd \bst = \int_{[0,1]^d} D f(\bst) \nabla_{\cP,\cA}(\bst) \dd \bst .
\end{eqnarray*}
Hence, according to Lemma~\ref{le:OrHo}, 
\begin{equation}\label{hoe:disc}
|L(f)| \le \|D f\|_{\Phi,\ast} \|\nabla_{\cP,\cA}\|_{\Phi}=\|f\|_{\cF_d} D_N^{\Phi}(\overline{\cP},\cA), 
\end{equation}
because after a change of variables $\bst \mapsto \boldsymbol{1}-\bst$ in the integral in the definition of the $\|\cdot\|_{\Phi}$-norm and because $\nabla_{\cP,\cA}(\boldsymbol{1}-\bst)=-\Delta_{\overline{\cP},\cA}(\bst)$ for almost all $\bst$, it follows that $$\|\nabla_{\cP,\cA}\|_{\Phi}=\|\Delta_{\overline{\cP},\cA}\|_{\Phi}=D_N^{\Phi}(\overline{\cP},\cA).$$ Hence $$e(Q_{\cP,\cA};\cF_d) \le D_N^{\Phi}(\overline{\cP},\cA).$$     

Define $$f_{\ast}(\bsx):=\frac{1}{n_{\nabla}} \int_{[\bszero,\bsx]} \mathrm{sgn}(\nabla_{\cP,\cA}(\bst))\,\Phi'\left(\frac{|\nabla_{\cP,\cA}(\bst)|}{\|\nabla_{\cP,\cA}\|_{\Phi}}\right) \dd \bst,$$ where $$n_{\nabla}:=\int_{[0,1]^d} \frac{|\nabla_{\cP,\cA}(\bst)|}{\|\nabla_{\cP,\cA}\|_{\Phi}} \Phi'\left(\frac{|\nabla_{\cP,\cA}(\bst)|}{\|\nabla_{\cP,\cA}\|_{\Phi}}\right) \dd \bst,$$ such that $L(f_\ast)=\|\nabla_{\cP,\cA}\|_{\Phi}$. Then $Df_{\ast}= \mathrm{sgn}(\nabla_{\cP, \cA}) \, g_{\nabla_{\cP,\cA}}$ is the signed version of the norming function from Lemma~\ref{le:OrHo} corresponding to the upper anchored local discrepancy function $\nabla_{\cP,\cA}$; it has $\|\cdot\|_{\Phi,\ast}$-norm equal to one and yields equality in the H\"older argument applied in \eqref{hoe:disc}. This finishes the proof of \eqref{disc:int:dual}.

For the initial error, it follows from the above that $$e(0;\cF_d)=\|\nabla_{\emptyset}\|_{\Phi}.$$ We have $\nabla_{\emptyset}(\bst)=\prod_{j=1}^d(1-t_j)$. Recall that $\Delta_{\emptyset}(\bst)= - \prod_{j=1}^d t_j$. Hence, again, a change of variables $\bst \mapsto \boldsymbol{1}-\bst$ in the integral in the definition of the $\|\cdot\|_{\Phi}$-norm shows that $\|\nabla_{\emptyset}\|_{\Phi}=\|\Delta_{\emptyset}\|_{\Phi}$. 
\end{proof}

\subsection*{Embeddings}

\paragraph{Proof of Remark~\ref{re:1}.} Let $r:=\min(p,2) \in (1,2]$. We show the existence of a number $c_r>0$ such that $\Phi(t) \le c_r t^r$ for all $t \ge 0$. For \(r=2\), the claim follows from \(\log(1+t)\le t\). For $r \in (1,2)$ we have $t^{1-r} \log(1+t) \le t^{2-r} \rightarrow 0$ as $t \to 0$, since $r<2$. Thus $\lim_{t \to 0} \Phi(t)/t^r=0$. On the other hand, $$\lim_{t \to \infty} \frac{\Phi(t)}{t^r}= \lim_{t \to \infty} t^{1-r} \log(1+t)=0,$$ since $r>1$. Since $t \mapsto \Phi(t)/t^r$ is continuous on $(0,\infty)$, this proves the claim. 

The assertion is trivial if $f=0$. Now let $0\ne f \in L_r$. Set $K=c_r^{1/r} \|f\|_{L_r}$. Then we have 
\begin{align*}
\int_{[0,1]^d} \Phi\left(\frac{|f(\bst)|}{K}\right) \dd \bst
&\le c_r \int_{[0,1]^d} \left(\frac{|f(\bst)|}{K}\right)^r \dd \bst\\
&=\frac{c_r}{K^r} \int_{[0,1]^d} |f(\bst)|^r \dd \bst
=\frac{c_r}{c_r \|f\|_{L_r}^r} \|f\|_{L_r}^r=1.
\end{align*}
Hence, $\|f\|_{\Phi} \le c_r^{1/r} \|f\|_{L_r}$.

Now, for $p \in (1,2]$ we have $r=p$ and hence $\|f\|_{\Phi}\le c_p^{1/p} \|f\|_{L_p}$. For $p >2$ we have $r=2$ and hence $\|f\|_{\Phi}\le c_2^{1/2} \|f\|_{L_2} \le c_2^{1/2} \|f\|_{L_p}$, due to the monotonicity of the $L_p$-norm. This finishes the proof. \hfill $\qed$

\paragraph{Declaration of generative AI use}

The authors used ChatGPT 5.6 Sol for literature searches, exploratory
development, and the preparation of portions of the exposition and
LaTeX source. All mathematical arguments, calculations, references, and conclusions were independently checked and verified by the authors, who take full responsibility for the contents of the paper.

\vspace{0.5cm}
\noindent{\bf Author's Addresses:}

\noindent Josef Dick, School of Mathematics and Statistics, The University of New South Wales, Sydney, NSW 2052, Australia. 

\noindent Email: josef.dick@unsw.edu.au \\

\noindent Friedrich Pillichshammer, Institut f\"{u}r Finanzmathematik und Angewandte Zahlentheorie, Universit\"{a}t Linz, Altenbergerstra{\ss}e 69, A-4040 Linz, Austria. 

\noindent Email: friedrich.pillichshammer@jku.at

\end{document}